\documentclass[12pt]{amsart}
\usepackage[a4paper, left=3.1cm, right=3.1cm, bottom=3cm,top=3cm]{geometry}
\usepackage{mathscinet}
\usepackage{amsmath, amsfonts, amssymb, amsthm}
\usepackage{xcolor}
\definecolor{ocean}{rgb}{0,0.5,0.5}
\definecolor{blue}{rgb}{0.00,0.26,0.50}
\usepackage[colorlinks = true, linkcolor=ocean, citecolor=blue]{hyperref}
\newtheorem{theorem}{Theorem}[section]
\newtheorem{proposition}[theorem]{Proposition}
\newtheorem{definition}[theorem]{Definition}
\newtheorem{remark}[theorem]{Remark}

\newtheorem{example}[theorem]{Example}

\newtheorem{corollary}[theorem]{Corollary}
\newtheorem{lemma}[theorem]{Lemma}

\def\<{\,<\!}
\def\>{\!>\,}

\def\F{\mathbb{F}}

\begin{document}
	
\title[Real and Rational elements]{On real and rational elements in a class of \\Lie groups}
	
\author{Arunava Mandal and Shashank Vikram Singh}

\address{Department of Mathematics,
Indian Institute of Technology Roorkee,
Uttarakhand 247667, India}
\email{arunava@ma.iitr.ac.in}

\address{Department of Mathematical Sciences,
Indian Institute of Science Education and Research Mohali,
Punjab 140306, India}
\email{shashank@iisermohali.ac.in}

	\begin{abstract}
   For a class of groups $G$ over a field $\mathbb{F}$, including certain Lie groups, Algebraic groups and finite groups, we provide a general method to determine real and rational elements, thereby unifying earlier group-specific results into a wider framework. As an application, we classify all real and rational elements in the semidirect product ${\rm SL}(2,\mathbb{R}) \ltimes \mathrm{Sym}^n(\mathbb{R}^2)$. Moreover, we study real elements in connected real solvable Lie groups. It is easy to observe that if $G$ is a simply connected solvable real Lie group, then the identity is the only real element of $G$. Conversely, if the identity is the only real element of a connected real Lie group $G$, then $G$ is simply connected, solvable, and exponential, which turns out to be of independent interest.


	\end{abstract}
\maketitle
    \subjclass{\it 2020 Mathematics Subject Classification:} Primary 20H20, 20E45; Secondary 20D10, 20G20 

\keywords{\bf Keywords:} Solvable Lie groups, Classical Algebraic groups, Affine groups, Real and rational elements.

\setcounter{tocdepth}{1}
\tableofcontents

\section{Introduction}
Let $G$ be a group.  
An element $g \in G$ is called \emph{real} if it is conjugate to its inverse in $G$, and \emph{rational} if it is conjugate to all powers $g^k$, where $\gcd(k,{\rm Ord}(g))=1$ whenever $g$ has finite order.  
These notions are naturally expressed in terms of conjugacy classes and play a key role in representation theory, especially through character theory.  
For a finite group $G$, an element $g$ is real precisely when $\chi(g) \in \mathbb{R}$ for every $\chi \in {\rm Irr}(G)$, the set of irreducible complex characters of $G$.  
Similarly, $g$ is rational if and only if $\chi(g) \in \mathbb{Q}$ for all $\chi \in {\rm Irr}(G)$.  
The correspondence between real conjugacy classes and real irreducible characters is well behaved and bijective (see for e.g. \cite{TZ05}), whereas the situation for rational conjugacy classes and rational-valued irreducible characters is far more subtle \cite{NT08}.  
This motivates the systematic study of real and rational elements across various families of groups, both finite and Lie groups (see \cite{OS15}, \cite{GM23}). The study of real elements-an algebraic point of view of what are commonly known as \emph{reversible elements}-has arisen in both mathematical and physical systems from several different directions (\cite{AA68}, \cite{D76}, \cite{L92}, \cite{OS15} and the references therein). The reversibility problem for the isometry group $O(n)\ltimes \mathbb{R}^n$ of $n$-dimensional Euclidean space was classified by Short \cite{SH08} and later extended in \cite{GL22} for the isometry group $U(n,\mathbb D)\ltimes \mathbb D^n$, where $\mathbb D=\mathbb C$ or $\mathbb H$, by Gangopadhyay and Lohan. In a similar direction, O'Farrell and Short asked the question of classifying reversible elements in the affine group $\mathrm{Aff}(n,\mathbb{R})$ \cite{OS15}. For the fields $\mathbb{F}=\mathbb{R}$ or $\mathbb{C}$, this question was subsequently answered for $\mathrm{Aff}(n,\mathbb{F})$ by Gangopadhyay, Maity, and Lohan \cite{GLM23}.

In this article, we provide a criterion for detecting real and rational elements in more general groups of the form
$G=H\ltimes N,$ where $H$ acts by automorphisms on a nilpotent normal subgroup $N$ of $G$.  
In particular, we show how information about real (resp.\ rational) elements in $H$ can often be lifted to obtain corresponding real (resp.\ rational) elements in $G$.  
This provides a unified framework that extends several earlier results obtained for special families of groups.
We illustrate our approach in particular for the groups ${\rm SL}(2,\mathbb{R}) \ltimes \mathbb{R}^n$ and affine groups ${\rm GL}(n,\mathbb{R}) \ltimes \mathbb{R}^n$.


There has been considerable work on real elements in various groups, see~\cite{OS15} and the references therein.  
For instance, real elements in algebraic groups and finite groups of Lie type have been studied in~\cite{TZ05,ST08}, primarily in the context of semisimple groups; much of the subsequent literature has continued along this direction.  
For affine groups, which possess an abelian nilradical, real elements have been investigated in~\cite{GLM23}.  
Recall that a subgroup $N$ of a Lie group $G$ is called the \emph{nilradical} if it is the maximal connected nilpotent normal subgroup of $G$.  
 In contrast to the semisimple case, much less is known about \emph{mixed type} Lie groups, i.e.\ those with both a non-trivial radical, the maximal connected nilpotent normal subgroup of $G$, and a semisimple part. A Lie group is called \emph{semisimple} if its radical is trivial.  
In this work, we focus on precisely such groups. We now recall some terminology and definitions needed for Theorem~\ref{theorem-F-nilpotent}.

\begin{definition}\label{def-F-nil}\cite{DM17}
Let $\mathbb{F}$ be a field. A \emph{nilpotent group} $N$ is said to be \emph{$\mathbb{F}$-nilpotent} if it satisfies the following two properties:
\begin{enumerate}
\item 
$N = N_0 \triangleright N_1 \triangleright \cdots \triangleright N_r = \{e\}$
 is the central series of $N$,

\item each quotient $N_j / N_{j+1}$ is a finite-dimensional vector space over $\mathbb{F}$.
\end{enumerate}
\end{definition}

In the following definition, let $N$ be an $\mathbb{F}$-nilpotent group with the central series
\[
N = N_0 \triangleright N_1 \triangleright \cdots \triangleright N_r = \{1\}.
\]
\begin{definition}\label{def-linear-action}\cite{DM17}
Suppose a group $G$ acts on $N$ by automorphisms. We say the action is \emph{$\mathbb{F}$-linear} if, for each index $j$, the induced action of $G$ on the quotient $N_j / N_{j+1}$ is $\mathbb{F}$-linear. Furthermore, for any element $x \in G$, we say that the action of $x$ on $N$ has no nonzero fixed points if, for each $j$, the action of $x$ on $N_j / N_{j+1}$ fixes only the zero element.
\end{definition}

It is worth noting that if $G$ is a group that acts by automorphisms on an $\mathbb{F}$-nilpotent normal subgroup $N$, then the action of $G$ on each $N_j / N_{j+1}$ factors through the quotient group $G/N$.

\begin{theorem}\label{theorem-F-nilpotent}
Let $\mathbb{F}$ be a field and let $N$ be a nilpotent normal $\mathbb{F}$-subgroup of $G$ such that $G = H \ltimes N$ for some subgroup $H \leq G$.  
Suppose the conjugation action of $G$ on $N$ is $\mathbb{F}$-linear.  
Let $x \in H$ be a non-identity element, and assume that the action of $x$ on $N$ has no nonzero fixed point. Then:
\begin{enumerate}
    \item if $x$ is real in $H$, then $xn$ is real in $G$ for all $n \in N$,
    \item if $x$ is rational in $H$, then $xn$ is rational in $G$ for all $n \in N$.
\end{enumerate}
\end{theorem}

Note that a large class of finite groups as well as Lie groups falls within the framework of Theorem~\ref{theorem-F-nilpotent}.  
For instance, if $G = (SA)\ltimes N$, where $S$ is semisimple, $A$ is abelian, and $N$ is simply connected and nilpotent Lie group, then $G$ satisfies the hypotheses of the theorem. We also give an example, in subsection \ref{example-remark}, showing that the hypothesis of Theorem~\ref{theorem-F-nilpotent} can be satisfied not only when the nilradical is abelian, but also when it is a higher-step nilpotent group. 
It is important to emphasize, however, that the condition in Theorem~\ref{theorem-F-nilpotent} is sufficient but not necessary; a counterexample can be found among connected solvable Lie groups (see Remark \ref{remark-sufficient}).  
We now restrict our attention to connected solvable Lie groups in order to obtain sharper results. Recall that an element $g \in G$ is called \emph{strongly real} if it can be expressed as a product of two involutions in $G$; equivalently, if there exists an involution $h \in G$ such that $g^{-1} = h g h^{-1}.$ For most classical groups over $\mathbb{R}$ or $\mathbb{C}$, it is known that a semisimple element is real if and only if it is strongly real~\cite{OS15}. The situation for unipotent elements—whether they are real or strongly real—has been addressed in~\cite{GM23}. In contrast, much less is known for mixed type of elements, i.e., it consists of both semisimple and unipotent part.  
In this work, we study the real and strongly real elements in a class of solvable Lie groups.

\begin{theorem}\label{a class of solvable groups}
Let $A$ be an abelian subgroup and $N$ be a simply connected nilpotent subgroup of $G$ such that 
$G = A \ltimes N$. Let $N= N_0 \triangleright N_1 \triangleright \cdots \triangleright N_r =\{e\},$
be the central series of $N$. Let $e\neq x\in A$. Then the following holds.
\begin{enumerate}
    \item If $x$ is conjugate to $x^{-1}$ in $G$, then $x^2 = e$.
    \item If $xn$ is real in $G$, then $x^2 = e$.
    \item If the action of $x$ on $N_j/N_{j+1}$ has no non-zero fixed point for each $j$, then 
    $xn =(xn)^{-1}$ for all $n \in N.$
    In particular, if $e\neq x \in A$ is real, then $xn$ is strongly real for all $n \in N$. In particular, $xn$ is real if and only if $xn$ is strongly real for all $n\in N$.
    
    \item Suppose the action of $A$ on $Z(N)$ is trivial, then $xn$ is not real in $AZ(N)$ for any $n\in Z(N)\setminus\{e\}$.
\end{enumerate}
\end{theorem}

\begin{remark}
\begin{enumerate}
    \item Although a simply connected nilpotent Lie group has no nontrivial real elements (cf. Lemma \ref{lamma-nilpotent-group-real-elements}), a nontrivial element of its nilradical may be real in a larger simply connected solvable Lie group. For example, let $G=\mathbb{R}^2\rtimes\mathbb{R}$, where $t\in\mathbb{R}$ acts on $\mathbb{R}^2$ by rotation through angle $t$. Then $G$ is simply connected and solvable, with nilradical $N=\mathbb{R}^2$. For $t=\pi$, conjugation by $(0,\pi)$ sends every $v\in\mathbb{R}^2$ to $-v$, and hence $(0,\pi)(v,0)(0,\pi)^{-1}=(-v,0)=(v,0)^{-1}$. Thus every nonzero element of $N$ is real in $G$, although no nontrivial element of $N$ is real with respect to conjugation inside $N$. This shows that, in the solvable case, real elements are controlled by the action of $G/N$ on the nilradical. Therefore, Theorem~\ref{a class of solvable groups} provides a method for constructing a class of solvable groups in which most elements are real.

\item In fact, the examples given above suggest the following. Let $G = H \ltimes_\sigma V$, where $V$ is a real vector space and $H$ acts linearly on $V$ and $\sigma$ is a conjugation action. Then every element of $V$ is real in $G$ if and only if, for every $v \in V$, there exists $h \in H$ such that $\sigma(h)(v) = -v$. 
In a similar spirit, if we consider a simply connected nilpotent group $N$ in place of $V$, we observe in Lemma \ref{prop-center-no-real} that when no nontrivial element of $Z(N)$ is real in $G$ (see also Examples \ref{examples-no-central-real}). 
\end{enumerate}
\end{remark}

In Lemma~\ref{lamma-nilpotent-group-real-elements}, we in fact showed that the set of real elements of a connected nilpotent Lie group $G$ is precisely $G[2]=\{g\in G:g^2=e\}$. In particular, $G[2]\subseteq Z(G)$, and if $G$ is simply connected and nilpotent, then the identity element is the only real element of $G$. We further observe that the latter conclusion holds for simply connected exponential solvable Lie groups (cf. Proposition~\ref{prop-sc-exp-solv}). Conversely, we show that if $G$ is a  connected Lie group, and the identity is the only real element of $G$, then $G$ is a simply connected exponential solvable Lie group (cf. Theorem \ref{theorem-trivial-real-element}); this can be viewed as an independent interest. Recall that a Lie group is called {\it exponential} if the exponential map $\exp:\mathfrak{g}\to G$ is surjective.


We then apply Theorem \ref{theorem-F-nilpotent} to an important class of semidirect products: the groups ${\rm SL}(2,\mathbb R)\ltimes \mathbb R^n$, where $\mathbb R^n$ is an algebraic representation space for ${\rm SL}(2, \mathbb R)$.
For each nonnegative integer $n$, the group ${\rm SL}(2,\mathbb{R})$ admits a unique (up to isomorphism) irreducible representation of dimension $n+1$. This representation is conveniently realized on the vector space of homogeneous polynomials of degree $n$ in two real variables. Explicitly, we define
$V_n \; \cong \; \mathrm{Sym}^n(\mathbb{R}^2) \; \cong \; \mathbb{R}^{n+1}$
as $$V_n = \left\{ p(x,y) \;=\; \sum_{i=0}^n a_i \, x^{\,n-i} y^{\,i} \; : \; a_i \in \mathbb{R} \right\}.$$
For $g = \begin{pmatrix} a & b \\ c & d \end{pmatrix} \in {\rm SL}(2,\mathbb{R})$ (with $ad-bc=1$), the action is given by
$$(g \cdot p)(x,y) \;=\; p(ax + by,\; cx + dy),$$
endowing $V_n$ with a linear representation $\rho: {\rm SL}(2,\mathbb{R}) \to {\rm GL}(V_n)$. This yields the unique irreducible finite-dimensional representation of ${\rm SL}(2,\mathbb{R})$ of dimension $n+1$.

\begin{theorem}\label{them-homogeneous-rep}
 Let $G= {\rm SL}(2,\mathbb{R}) \ltimes V_{n}$. Then the following holds.
\begin{enumerate}
\item  For even $n$:
    \begin{enumerate}
       \item if $x = \pm I_{2\times 2}$, then $(x,v)\in G$ is real if and only if there exists $y \in {\rm SL}(2,\mathbb{R})$ satisfying $\rho(y)v=-v$,
        \item if $x \neq \pm I_{2\times 2}$  and $x$ is real, then $(x,v)$ is real for all $v \in V_{n}$.
        \item $(\pm I_{2\times 2}, v)$ is not strongly real for all $0\neq v \in V_n.$
    \end{enumerate}
\item For odd $n$:
    \begin{enumerate}
       \item if $x= I_{2\times 2}$, then $(I_{2\times 2},v)$ is real for all $v \in V_n$,
        \item if $x \neq I_{2\times 2}$  and $x$ is real, then $(x,v)$ is real for all $v \in V_{n}$.

        \item $(\pm I_{2\times 2}, v)$ is strongly real for all $ v \in V_n.$
    \end{enumerate}
\end{enumerate}
   \end{theorem}

   In \cite{GLM23}, it is shown that for affine groups $G = {\rm GL}(n,\mathbb{F}) \ltimes \mathbb{F}^n$, with $\mathbb{F} = \mathbb{R}, \mathbb{C}$, if $x \in {\rm GL}(n,\mathbb{F})$ is real, then $(x,v)$ is real for every $v \in \mathbb{F}^n$. When $x$ satisfies the condition of Theorem~\ref{theorem-F-nilpotent}, a part of their proof (cf.\ Lemma~3.4 in \cite{GLM23}) follows directly from Proposition~\ref{vector group}, without the need to compute explicit conjugating elements for $(x,v)$. Furthermore, we extend this result to establish the analogous statement for \emph{rational} elements in these affine groups in Corollary \ref{cor-affine}.

The paper is organized as follows. In \S2 we recall basic definitions and preliminary facts about real and rational elements and deduce Proposition \ref{vector group}. In \S3 we prove Theorem \ref{theorem-F-nilpotent}. Theorem \ref{theorem-trivial-real-element} and Theorem \ref{a class of solvable groups} are proved in \S4. In \S5 we apply results on ${\rm SL}(2,\mathbb R)\ltimes \mathbb R^n$, and affine groups ${\rm GL}(n,\mathbb R)\ltimes \mathbb R^n$. 

\section{Preliminaries}
In this section, we recall the basic definitions, fix our notation, and present the key observations in Proposition \ref{vector group}.
Let $G$ be a group and let $x \in G$ be an element. We denote by $\langle x \rangle$ the cyclic subgroup generated by $x$ in $G$, and by ${\rm Ord}(x)$ the order of $x$. 

\subsection{Rational elements}
We now recall the definition of a rational element in $G$ from \cite[Section~5]{NT08}.  

\begin{definition}\label{def-rational-element}
An element $x \in G$ is called a \emph{rational element} if whenever $\langle y \rangle = \langle x \rangle$, the element $y$ is $G$-conjugate to $x$. Equivalently, for every integer $k$ relatively prime to ${\rm Ord}(x)$, the element $x$ is conjugate to $x^k$.
\end{definition}

Suppose $G$ is a finite group. Then Definition~\ref{def-rational-element} is equivalent to the following: an element $x \in G$ is called a \emph{rational element} if, for every integer $k$ relatively prime to the order of $G$, the element $x$ is conjugate to $x^k$. 
Note that the conjugate of a rational element is again a rational element.  
We denote by ${\rm Cl}(x)$ the conjugacy class of $x$.  
It is easy to see that if $x \in G$ is rational, then $xN$ is rational for any normal subgroup $N$ of $G$.  
Moreover, if $x \in G$ is such that $\gcd({\rm Ord}(x), {\rm Ord}(N)) = 1$, then $xN$ being rational implies that $x$ is rational; see Lemma~5.1(e) of \cite{NT08}. 
The following result is in a similar spirit, but in a more general setting; in particular, we consider groups that include certain finite groups, algebraic groups and more generally, a class of Lie groups. Part~$(1)$ of the proof of Proposition \ref{vector group} is contained in Proposition 2.1 of \cite{DM17}. For completeness, we include a detailed proof here. 

\begin{proposition}\label{vector group}
Let $G = H \ltimes V$, where $V$ is an $\mathbb{F}$-vector subgroup and the conjugation action of $H$ on $V$ is $\mathbb{F}$-linear.  
Let $e\neq x \in H$ and suppose that the action of $x$ on $V$ has no nonzero fixed point. Then:
\begin{enumerate}
    \item for every $v \in V$, there exists a unique $w \in V$ (depending on $v$) such that  
    $x v = w x w^{-1}$,
    \item if $x$ is real in $H$, then $xv$ is real in $G$ for all $v \in V$,
    \item if $x$ is rational in $H$, then $xv$ is rational in $G$ for all $v \in V$.
\end{enumerate}
\end{proposition}

\begin{proof}
   $(1)$ Let $\sigma$ be the conjugation action of $H$ on $V$. Since the conjugation action of $x$ on $V$ has no nonzero fixed points, it follows that $1$ is not an eigenvalue of the linear operator $\sigma(x)$. Thus the operator 
    $\sigma(x)^{-1}-I:V\to V$
 is invertible. Indeed, if there is a non-zero $v\in V$ such that $v$ is in the kernel of  
$\sigma(x)^{-1}-I$, then $v$ is a nontrivial fixed point under conjugation by $x$, contradicting the hypothesis.
Thus the kernel is trivial, so the map is injective, and since $V$ is finite‑dimensional, $\sigma(x)^{-1}-I$ is also surjective.
 Hence, for any vector $v\in V$, there exists a unique $w$ in $V$ such that $v=(\sigma(x)^{-1}-I)w.$ Equivalently, $v=x^{-1}wx-w$, which, in turn, in multiplicative notation, $xv=wxw^{-1}.$

    Now assume for $v\in V$,
     $xv=w_1xw_1^{-1}=w_2xw_2^{-1}$ for some $w_1, w_2\in V.$ Then $w_2^{-1}w_1$ commutes with $x$, i.e., it is a fixed point under the conjugation action of $x$.
     By hypothesis, the only fixed point is the zero vector, so $w_2^{-1}w_1=0$, forcing
      $w_1=w_2.$ This shows uniqueness of the choice of $w$, completing the proof of statement $(1)$.

   $(2)$ Let $h\in H$ be such that $x=hx^{-1}h^{-1}$. Then
    $xv=whx^{-1}h^{-1}w^{-1}$ and 
    \begin{align*} (xv)^{-1}&=(wxw^{-1})^{-1}=wx^{-1}w^{-1}\\&=w(wh)^{-1}(wh)x^{-1}(wh)^{-1}(wh)w^{-1}\\&=(wh^{-1}w^{-1})(xv)(wh^{-1}w^{-1})^{-1}.
    \end{align*}
    Therefore, $xv$ is conjugate to $(xv)^{-1}.$ This proves $(2).$

    $(3)$ Given that for some $h \in H$, $x=hx^{k}h^{-1}$. Then $xv=wxw^{-1}=whx^{k}h^{-1}w^{-1}$ and 
    \begin{align*}
        (xv)^k&=wx^kw^{-1}=wh^{-1}xhw^{-1}\\
        &=wh^{-1}(w^{-1}w)x(w^{-1}w)hw^{-1}\\
        &=(wh^{-1}w^{-1})(wxw^{-1})(whw^{-1})\\
        &=(wh^{-1}w^{-1})(xv)(wh^{-1}w^{-1})^{-1}.
          \end{align*}
    Therefore, $xv$ is conjugate to $(xv)^{k}.$ Note that if ${\rm Ord}(x)$ is $n$, then we observe that
    ${\rm Ord}(xv)$ divides ${\rm Ord}(x)$. Indeed, $(xv)^n=(wxw^{-1})^n=wx^nw^{-1}=e.$
    Therefore, it follows that $xv$ is conjugate to $(xv)^r$ for all $r$ co-prime to ${\rm Ord}(xv)$.
    This proves $(3).$
\end{proof}

\begin{remark}
   Proposition~\ref{vector group} shows that if $G/V$ contains a rational element $xV$ of order $m$, and $x$ satisfies the hypothesis of Proposition~\ref{vector group}, then $G$ contains a rational element of order $m$. One can see Lemma 5.2 of \cite{NT08} in a similar direction for the finite group context.
\end{remark}

\subsection{$\mathbb F$-nilpotent group}

 We note an observation in the following lemma, which will be used to prove Theorem \ref{theorem-F-nilpotent}. 

\begin{lemma}\label{lemma-same-action}
Let $G$ be a group with an $\mathbb{F}$-nilpotent normal subgroup $N$ of $G$ such that the action of $G$ on $N$ is linear. Let $x \in G$.  
Then, for each $j$, the action of $x$ on $N_j/N_{j+1}$ is same as the action of $xN$ on $N_j/N_{j+1}$. \qed
\end{lemma}
\begin{proof}
  Since the action of $N$ on $N_j/N_{j+1}$ is trivial, the action of $x$ on each $N_j/N_{j+1}$ coincide with the action of $xN$ on $N_j/N_{j+1}$. 
\end{proof}
Note that, if $N$ is a simply connected real Lie group, then it satisfies all the properties in Definition~\ref{def-F-nil}, and hence is an $\mathbb{R}$-nilpotent subgroup. Also, if $N$ is a nilpotent algebraic group defined over a field $\mathbb F$, then it is $\mathbb F$-nilpotent.
\section{Groups over a field $ \mathbb F$}
In this section, we prove Theorem~\ref{theorem-F-nilpotent} and illustrate, with an example, how it can be used to compute real and rational elements in a certain type of finite group.

\vspace{.2cm}
Proof of Theorem \ref{theorem-F-nilpotent}:
  Let \(x \in H\) be real and let \(n \in N\). We proceed inductively through the central series
\[
N = N_0 \triangleright N_1 \triangleright \cdots \triangleright N_r = \{e\}.
\] We consider the coset $xnN_1$
  inside the product $<x>\cdot N/N_1.$ 
   Since the action of $x$ has no non-zero fixed point in $N/N_1$, Proposition \ref{vector group}, guarantees that $xnN_1$ is conjugate to its inverse. Explicitly, there exists $w\in N/N_1$ such that $xnN_1=wxw^{-1}N_1.$
   Hence, we can write
   $xn=x_1n_1,$ with $x_1=wxw^{-1}$ and $n_1\in N_1.$
    Notice that, by Lemma \ref{lemma-same-action}, $x_1$ has the same action on $N_1/N_2$ as 
$x$ does, and since $x$ is real (conjugate to its inverse), so is $x_1$. Therefore, Proposition \ref{vector group} applies again to $x_1n_1N_2$ in $<x_1>\cdot N_1/N_2$ showing it is conjugate to its inverse. That yields $x_1n_1=x_2n_2,$ where $x_2=w_1x_1w_1^{-1}$ for some $w_1\in N$, and $n_2\in N_2$. Note that, by Lemma \ref{lemma-same-action} the action of $x_2$ on $N_2/N_3$ has no non-zero fixed point, and $x_2$ is real, so we repeat this process down the chain of successive quotients.
Suppose at stage $j$, we have $x_{j-1} n_{j-1} = x_j n_j,$
where \(x_j = w_{j-1} x_{j-1} w_{j-1}^{-1}$, \(n_j \in N_j\), and \(x_j\) is real. Since the action of \(x_j\) on \(N_j / N_{j+1}\) has no non-zero fixed point, by applying Proposition \ref{vector group} we have $x_j n_j = x_{j+1} n_{j+1}$ for some $x_{j+1} = w_j\, x_j\, w_j^{-1}$ and \(n_{j+1} \in N_{j+1}\), and so \(x_{j+1}\) remains real.
Continuing this process down the chain yields
\[
x n = x_1 n_1 = x_2 n_2 = \cdots=x_{r-1}n_{r-1} = x_r n_r = x_r,
\]
where $n_r\in N_r=\{e\}$, and $x_r$ is real. Therefore \(x n\) is real. Since \(n \in N\) is arbitrary, this proves that $xn$ is real for all $n\in N$, as required.  

$(2)$ The proof proceeds just as in case $(1).$ The key point is: because $x$ is a rational element in 
$H$, and since $xn=x_1n_1$ with $x_1=wxw^{-1}$, it immediately follows that $x_1$ is rational. In the same way, each $x_j$ defined as in the equation above must also be rational. Therefore, for every $v\in V$, the element $xv$ remains rational.
\qed

The following example demonstrates that the above theorem is also very helpful from a computational viewpoint. 
\begin{example}
    Let $G={\rm PSL}(2,\mathbb{Z}_2) \ltimes \mathbb Z_2^2$. The order of the group is $ 6 \times 4 = 24$. The group ${\rm PSL}(2,\mathbb{Z}_2)$ have three rational classes, namely, $\begin{pmatrix}
        1 & 0\\
        0 & 1
    \end{pmatrix},$ $\begin{pmatrix}
        1 & 1\\
        0 & 1
    \end{pmatrix}$ and $\begin{pmatrix}
        1 & 1\\
        1 & 0
    \end{pmatrix}$. Note that $x= \begin{pmatrix}
        1 & 1\\
        1 & 0
    \end{pmatrix}$ has order three and $x$ is conjugate to $x^2$.     We check the reality/rationality of $xv$ by the following computations, let $xv=\begin{pmatrix}
        1 & 1 & v_1\\
        1 & 0 & v_2\\
        0 & 0 & 1
    \end{pmatrix}.$ The order of $xv$ is three, therefore $(xv)^k$ is either $xv$ or $(xv)^2$. The $k=1$ case is trivial, and for the other case,  $xv$ is conjugate to $(xv)^2$ by $\begin{pmatrix}
        0 & 1 & v_2\\
        1 & 0 & v_2\\
        0 & 0 & 1
    \end{pmatrix}.$
    
    On the other hand, it is easy to see that $x$ has no non-zero fixed point in $\mathbb Z_2^2$. Therefore, the same conclusion follows directly from Theorem \ref{theorem-F-nilpotent}. Computing the conjugating matrix explicitly is not required.
\end{example}
\subsection{Example and Remark}\label{example-remark}
We give an example showing that the hypothesis of Theorem~\ref{theorem-F-nilpotent} can be satisfied not only when the nilradical is abelian, but also when it is a higher-step nilpotent group. For this purpose, consider the group $G=\mathrm{GSp}(4,\mathbb R)\ltimes H$, where $H$ is a Heisenberg group, and
\[
\mathrm{GSp}(4,\mathbb R) = \left\{g\in {\rm GL}(4, \mathbb R): g^T J g = \mu(g)\, J,\ \mu(g)\in\mathbb R^\times\right\},
\]
with \(\mu(g)\) is the \emph{similitude factor} and \(J\) is the standard symplectic form on \(\mathbb R^4\) defined as
\[
J = \begin{pmatrix}0 & I_2 \\ -I_2 & 0\end{pmatrix}.
\]
 
The group $\mathrm{GSp}(4,\mathbb R)$ acts naturally by automorphisms on the  $5$-dimensional Heisenberg group $H$, where as a manifold $H$ is equal to $\mathbb R^4\times\mathbb R$ with elements written $(v,t)$.  
Its group operation is
\[
(v,t)\cdot (v',t') = \left(v+v',\ t+t' + \tfrac12\,\omega(v,v')\right),
\]
where \(\omega(v,v') = v^T J v'\). Then the center of $H$, 
\(Z(H)=\{(0,t)\}\cong\mathbb R\), and \(H/Z(H)\cong\mathbb R^4\).

\medskip

If \(g\in\mathrm{GSp}(4,\mathbb R)\) with \(g^T J g = \mu(g)\,J\), it induces an automorphism
\[
\varphi_g: (v,t)\mapsto \big(g v,\ \mu(g)\, t\big).
\]
 This defines the semidirect product $G= \mathrm{GSp}(4,\mathbb R) \ltimes H$.

\medskip

Consider the matrix
\[
x = \begin{pmatrix}P & 0 \\ 0 & P^{-1}\end{pmatrix}\in \mathrm{GSp}(4,\mathbb R), \text{ where } P = \begin{pmatrix}0 & 1\\ -1 & 0\end{pmatrix}.
\]
Since \(P^{-1}=-P\), we have \(x = \operatorname{diag}(P, -P)\). A direct computation shows
\[
x^TJx=-J,
\]
so \(\mu(x)=-1\). The matrix $x$ is real as it is conjugate to its inverse by $y=\begin{pmatrix}
    0 & I_2\\
    I_2 & 0
    \end{pmatrix} \in \mathrm{GSp}(4,\mathbb R). $

\medskip

It is easy to see that the eigenvalues of \(P\) are \(\pm i\). Hence \(x\) has eigenvalues \(\pm i\) with algebraic multiplicity two, so \(1\) is not an eigenvalue.  
Therefore \(x\) acts without nonzero fixed points on the quotient \(H/Z(H)\). Also, the action of $x$ on \(Z(H)\) is multiplication by \(\mu(x)=-1\). Therefore, by Theorem~\ref{theorem-F-nilpotent}, we conclude that $xn$ is real in $G$ for all $n \in H$. This completes our goal. It is possible to give the above example for $G={\rm GSp}(n,\mathbb R)\ltimes H$, however, we are not going into it here.


\section{Solvable Lie groups}
In this section, we prove Theorem \ref{theorem-trivial-real-element} and Theorem \ref{a class of solvable groups}. Additionally, we provide examples and remarks illustrating that the condition given in Theorem \ref{a class of solvable groups} is sufficient, but not necessary. First, we discuss several subcases. Although, some of these results may be easy to see, we include them for completeness.


\begin{lemma}\label{lamma-nilpotent-group-real-elements}
Let $G$ be a connected nilpotent Lie group. Then an element $g\in G$ is real if and only if $g^2=e$. Equivalently, the set of real elements of $G$ is precisely $G[2]=\{g\in G:g^2=e\}$. In particular, $G[2]\subset Z(G)$, and if $G$ is simply connected nilpotent then identity element is  the only real element in $G$.
\end{lemma}
\begin{proof}
If $g^2=e$, then $g=g^{-1}$, so $g$ is real. Conversely, suppose that $g$ is real, so there exists $a\in G$ such that $aga^{-1}=g^{-1}$. Let $p:\widetilde G\to G$ be the universal covering homomorphism. Since $p$ is surjective, choose $h\in\widetilde G$ with $p(h)=a$, and write $g=p(x)$ with $x=\exp X$, where $X \in \widetilde{\mathfrak g}$, the Lie algebra of $\widetilde{G}$. Then $p(hxh^{-1})=p(h)p(x)p(h)^{-1}=aga^{-1}=g^{-1}=p(x^{-1})$, so $hxh^{-1}=x^{-1}\gamma$ for some $\gamma\in\ker p$. Since $\widetilde G$ is simply connected nilpotent, $\ker p\subset Z(\widetilde G)$; write $\gamma=\exp Z$ with $Z\in Z(\widetilde{\mathfrak g})$, and $\operatorname{Ad}(h)X=-X+Z$. Put $N=\operatorname{Ad}(h)-I$. Since $\operatorname{Ad}(h)$ is unipotent, $N$ is nilpotent, while $NZ=0$. 
Hence $(N+2I)X=Z$, and applying $N$ gives $N(N+2I)X=0$. Since $N+2I$ is invertible, $NX=0$, so $2X=Z$. Therefore $x^2=\exp(2X)=\exp Z=\gamma\in\ker p$, and consequently $g^2=p(x^2)=e$.

For $g\in G[2]$, we have $\operatorname{Ad}(g)^2=I$. Since $G$ is connected and nilpotent, $\operatorname{Ad}(g)$ is unipotent; hence $\operatorname{Ad}(g)=I$. Therefore $g\in Z(G)$.
Since there is no finite order element in a simply connected nilpotent group, the last statement follows.
\end{proof}

Let $G$ be a Lie group with Lie algebra $\mathfrak g$. A connected Lie group is called exponential if the exponential map $\exp:\mathfrak g\to G$ is surjective. It is well known that a unipotent algebraic group has no nontrivial real elements (cf. Lemma 4.6 in \cite{pc}). We extend this result, in a similar spirit, to simply connected exponential solvable Lie groups.

\begin{proposition}\label{prop-sc-exp-solv}
Let $G$ be a simply connected exponential solvable Lie group. Then the identity element is the only real element of $G$.
\end{proposition}
\begin{proof}
Let $g\in G$ be real. Since $G$ is exponential, there exists a unique $X\in\mathfrak g$ such that $g=\exp X$. Choose $h\in G$ such that $hgh^{-1}=g^{-1}$, and write $h=\exp Y$. Then $\exp(\operatorname{Ad}(h)X)=\exp(-X)$, and injectivity of $\exp$ gives $\operatorname{Ad}(h)X=-X$. Since $\operatorname{Ad}(h)=\exp(\operatorname{ad}Y)$, this implies that $-1$ is an eigenvalue of $\exp(\operatorname{ad}Y)$ whenever $X\neq0$. Hence $\operatorname{ad}Y$ has an eigenvalue $(2k+1)\pi i$ for some $k\in\mathbb Z$, contradicting Dixmier's criterion (cf. \cite{dix}) for exponential solvable Lie algebras, which states that $\operatorname{ad}Y$ has no nonzero purely imaginary eigenvalues. Thus $X=0$, and consequently $g=e$.
\end{proof}

More generally, we prove the following result.

\begin{theorem}\label{theorem-trivial-real-element}
  Let $G$ be a connected (real) Lie group. Then the following conditions are equivalent:
  \begin{enumerate}
      \item $\{e\}$ is the only real element of $G$.

      \item $G$ is simply connected solvable and exponential.
  \end{enumerate}
\end{theorem}

\begin{proof}
$(1)\Rightarrow(2)$: Suppose that $G$ is not solvable. Then there exists a nontrivial semisimple Lie subgroup $S$ of $G$. Let $\mathfrak{s}$ be the Lie algebra of $S$. We claim that there exist $X\in\mathfrak{s}\setminus\{0\}$ and $s\in S$ such that $\operatorname{Ad}(s)X=-X.$
It is enough to consider a nontrivial simple factor of $\mathfrak{s}$.

First suppose that this simple factor is noncompact. Choose a Cartan decomposition
$\mathfrak{s}=\mathfrak{k}\oplus\mathfrak{p}$
and let $\mathfrak{a}\subset\mathfrak{p}$ be a maximal abelian subspace (for more details on the Cartan decomposition, see Section 2 of \cite{kms}). Since $\mathfrak{s}$ is noncompact and semisimple, the corresponding restricted root system $\Sigma\subset\mathfrak{a}^*$(dual of $\mathfrak{a}$) is nonempty. Choose a restricted root $\alpha\in\Sigma$. Let $H_\alpha\in\mathfrak{a}$ be the corresponding restricted coroot. The reflection
$r_\alpha(H)=H-\alpha(H)H_\alpha$
belongs to the restricted Weyl group
$W=N_K(\mathfrak{a})/Z_K(\mathfrak{a}),$
and, since $\alpha(H_\alpha)=2$, we have
$r_\alpha(H_\alpha)=-H_\alpha.$
Because the restricted Weyl group is realized by elements of $K$, there exists $s\in K\subset S$ such that
$\operatorname{Ad}(s)|_{\mathfrak{a}}=r_\alpha$ (see Theorem 6.57 of \cite{Kn}).
Hence, with $X=H_\alpha\neq0$,
$\operatorname{Ad}(s)X=-X.$

Now suppose that the simple factor is compact. Choose a maximal toral subalgebra $\mathfrak{t}$ of $\mathfrak{s}$. Its complexification $\mathfrak{t}_{\mathbb C}$ is a Cartan subalgebra of $\mathfrak{s}_{\mathbb C}$, and the corresponding root system is nonempty. Choose a root $\alpha$. The compact real form contains the corresponding real $\mathfrak{su}(2)$-subalgebra. Inside this $\mathfrak{su}(2)$-subalgebra, the Weyl-group representative corresponding to $\alpha$ is an element $s\in S$ satisfying
$\operatorname{Ad}(s)H_\alpha=-H_\alpha$ for a nonzero $H_\alpha\in\mathfrak{t}$. Thus again $\operatorname{Ad}(s)X=-X$
for some $X\in\mathfrak{s}\setminus\{0\}$.

Therefore, for every nontrivial connected semisimple real Lie group $S$, there exist $X\neq0$ and $s\in S$ such that
$\operatorname{Ad}(s)X=-X.$
Consequently, for sufficiently small $t\neq0$,
$s\exp(tX)s^{-1}
=\exp(t\operatorname{Ad}(s)X)
=\exp(-tX)
=\exp(tX)^{-1}.
$
Since $\exp(tX)\neq e$ for sufficiently small $t\neq0$, the element $\exp(tX)$ is a nontrivial real element of $S$. This proves that $G$ is a solvable Lie group.

Now we show that $G$ is simply connected. Let $K$ be a maximal compact subgroup of $G$. Since $G$ is solvable, $K$ is compact, connected, and abelian (cf. Theorem 13, Lemma 2.2 of \cite{I49}), and the inclusion $K\hookrightarrow G$ induces an isomorphism $\pi_1(K)\cong\pi_1(G)$, since any connected Lie group is topologically direct product to a compact subgroup and euclidean spaces (cf. \cite{Mo62}). Thus, if $G$ is not simply connected, then $K$ is of positive dimension. Since $K$ is a positive-dimensional torus, it contains a nontrivial element $k$ of order $2$. Hence $k=k^{-1}$, so $k$ is real, contradicting our assumption that $e$ is the only real element. Therefore $G$ is simply connected.

Now, suppose that $G$ is not exponential. Let $\mathfrak g$ be the Lie algebra of $G$. Since $G$ is connected, simply connected, and solvable, the Dixmier criterion (cf. \cite{dix}) states that $G$ is exponential if and only if $\operatorname{Spec}(\operatorname{ad}X)\cap i\mathbb R={0}$ for every $X\in\mathfrak g$. Since $G$ is not exponential, there exist $X\in\mathfrak g$ and $\beta\in\mathbb R\setminus\{0\}$ such that $i\beta$ is an eigenvalue of $\operatorname{ad}X$ on the complexification $\mathfrak g_{\mathbb C}$. Choose a (nonzero) eigenvector $Z\in\mathfrak g_{\mathbb C}$ such that $Z=Y+iW$, with $Y,W\in\mathfrak g$, and $\operatorname{ad}X(Z)=i\beta Z$. This gives $\operatorname{ad}X(Y)=-\beta W$ and $\operatorname{ad}X(W)=\beta Y$. Hence $V=\operatorname{span}_{\mathbb R}\{Y,W\}$ is an $\operatorname{ad}X$-invariant real subspace, and, with respect to the basis $\{Y,W\}$, $\operatorname{ad}X|_V=\begin{pmatrix}
 0&\beta\\-\beta&0\end{pmatrix}$. Therefore $\exp\left(\frac{\pi}{\beta}\operatorname{ad}X\right)|_V=-I$. In particular, $\operatorname{Ad}\left(\exp\left(\frac{\pi}{\beta}X\right)\right)Y=-Y$. Now set $g=\exp(Y)$ and $h=\exp\left(\frac{\pi}{\beta}X\right)$. Using $h\exp(Y)h^{-1}=\exp(\operatorname{Ad}(h)Y)$, we obtain $hgh^{-1}=\exp(-Y)=g^{-1}$. Thus $g$ is a real element of $G$. 
 Finally, $g\neq e$. Indeed, a connected, simply connected solvable Lie group has no nontrivial compact subgroups. If $\exp(Y)=e$ for some $Y\neq0$, then the one-parameter subgroup ${\exp(tY)\in\mathbb R}$ would be periodic and hence would contain a nontrivial compact subgroup isomorphic to $S^1$, which is impossible. Thus $Y\neq0$ implies $\exp(Y)\neq e$. Consequently, $g=\exp(Y)$ is a nontrivial real element of $G$, which is a contradiction. This proves that $G$ is exponential, as required.

 $(2)\Rightarrow(1)$: Follows from Proposition \ref{prop-sc-exp-solv}.
\end{proof}

\begin{lemma}
Let $G$ be a simply connected solvable Lie group with nilradical $N$. Then every real element of $G$ belongs to $N$. Moreover, since $N$ is simply connected nilpotent, an element $n=\exp X\in N$ is real in $G$ if and only if $-X\in\operatorname{Ad}(G)X$.
\end{lemma}
\begin{proof}
Since $G$ is simply connected and solvable, $G/N$ is a simply connected abelian Lie group, hence a vector group and therefore torsion-free. If $g\in G$ is real, then for some $h\in G$ we have $hgh^{-1}=g^{-1}$. Passing to $G/N$, which is abelian, gives $gN=(gN)^{-1}$, and hence $(gN)^2=N$. Since $G/N$ is torsion-free, $gN=N$, so $g\in N$. Now $N$ is simply connected nilpotent, so $\exp:\mathfrak n\to N$ is a diffeomorphism. Writing $n=\exp X$, we have $n$ real in $G$ if and only if there exists $h\in G$ such that $hnh^{-1}=n^{-1}$. Since $h\exp(X)h^{-1}=\exp(\operatorname{Ad}(h)X)$ and $n^{-1}=\exp(-X)$, injectivity of $\exp$ on $N$ gives $\operatorname{Ad}(h)X=-X$, which is equivalent to $-X\in\operatorname{Ad}(G)X$.
\end{proof}

\begin{lemma}\label{prop-center-no-real}
 Let $G=H\ltimes N$, where $N$ is a simply connected nilpotent Lie group, and let $H$ act on $N$ by automorphisms. Let $\mathfrak z(\mathfrak n)$ be the Lie algebra of $Z(N)$. For $z=\exp X\in Z(N)$, with $X\in\mathfrak z(\mathfrak n)$, the element $z$ is real in $G$ if and only if $-X\in H\cdot X$. In particular, if $H\cdot X$ does not contain $-X$ for every nonzero $X\in\mathfrak z(\mathfrak n)$, then no nontrivial element of $Z(N)$ is real in $G$.
\end{lemma} 
\begin{proof}
    Let $z=\exp X\in Z(N)$. If $z$ is real in $G$, there exists $g\in G$ such that $gzg^{-1}=z^{-1}$. Write $g=hn$, with $h\in H$ and $n\in N$. Since $z\in Z(N)$, we have $nzn^{-1}=z$, and hence $gzg^{-1}=hzh^{-1}$. Since $N$ is simply connected and nilpotent, the exponential map $\exp:\mathfrak n\to N$ is a diffeomorphism, so $hzh^{-1}=\exp(\operatorname{Ad}(h)X)$ and $z^{-1}=\exp(-X)$. Thus $\operatorname{Ad}(h)X=-X$, and consequently $-X\in H\cdot X$. Conversely, if $-X=\operatorname{Ad}(h)X$ for some $h\in H$, then $hzh^{-1}=\exp(-X)=z^{-1}$, so $z$ is real in $G$. 
\end{proof}

\begin{example}\label{examples-no-central-real}
The following semidirect products satisfy the hypothesis of Lemma~\ref{prop-center-no-real}.

\begin{enumerate}
    \item Let $G=\mathbb{R}_{>0}\ltimes \mathbb{R}^n$, where
    $t\cdot v=tv$. Since
    $H\cdot v=\{tv:t>0\}$, we have
    $-v\notin H\cdot v$ for every $v\neq0$.

    \item Let G=$\mathbb{R}_{>0}\ltimes H_{2m+1}$, where $H_{2m+1}$ is the
    Heisenberg group and the dilation is
    $\delta_t(v,z)=(tv,t^2z)$, where $(v,z)\in \mathbb R^{2m}\times \mathbb R$, and as a manifold $H_{2m+1}$ is equal to $\mathbb R^{2m}\times \mathbb R$, and $t\in \mathbb{R}_{>0}$. The center is one-dimensional, and
    $\delta_t(z)=t^2z$; hence
    $-z\notin H\cdot z$ for all $z\neq0$.

    \item Let $\mathfrak e=\operatorname{span}\{X,Y,Z,W\}$ be the Engel Lie algebra with nonzero brackets $[X,Y]=Z$ and $[X,Z]=W$. The grading $\mathfrak e=V_1\oplus V_2\oplus V_3$, where $V_1=\operatorname{span}\{X,Y\}$, $V_2=\mathbb RZ$, and $V_3=\mathbb RW$, defines, for each $t>0$, the Lie algebra automorphism $\delta_t:\mathfrak e\to\mathfrak e$ by $\delta_t(X)=tX,\ \delta_t(Y)=tY,\ \delta_t(Z)=t^2Z,\ \delta_t(W)=t^3W$. Let $E$ be the simply connected Engel group with Lie algebra $\mathfrak e$. Since $\exp:\mathfrak e\to E$ is a diffeomorphism, $\delta_t$ induce uniquely to a Lie group automorphism, also denoted by $\delta_t:E\to E$, satisfying $\delta_t\circ\exp=\exp\circ\delta_t$. Now let $G=\mathbb R_{>0}\ltimes E$, where $s\in\mathbb R_{>0}$ acts on $E$ by the automorphism $\delta_s$.
    Since $Z(E)=\{\exp(dW):d\in\mathbb R\}$ and $\delta_t(W)=t^3W$, we have $\delta_t(\exp(dW))=\exp(t^3dW).$

Let $z=\exp(dW)\in Z(E)$ is real in $G$.
Then there exists $(t,e)\in \mathbb R_{>0}\ltimes E$ such that $(t,e)z(t,e)^{-1}=z^{-1}.$
Since $z\in Z(E)$, conjugation by $e\in E$ is trivial, and hence $(t,e)z(t,e)^{-1}=\delta_t(z).$
Therefore, $\delta_t(z)=z^{-1},$
which is equivalent to $t^3d=-d.$
Since $t>0$, we have $t^3>0$, so the equation $t^3=-1$ has no solution. Thus $d=0$, and therefore $z=e$.
\end{enumerate}

Therefore, in each case $(1)-(3)$, if $N$ is the nilradical of $G$, then
$Z(N)\setminus\{e\}$ contains no real elements of $G$.
\end{example}

Now we proceed to proof Theorem \ref{a class of solvable groups}, for which we require the following lemma.


\begin{lemma}\label{AV}
Let $A$ be an abelian subgroup and $V$ a vector normal subgroup of $G$ such that $G = A\ltimes V$.  
Suppose that the action of $A$ on $V$ has no nonzero fixed point. Let $x\in A$.  
If $x$ is real in $G$ and the action of $x$ on $V$ has no nonzero fixed point, then $xv = (xv)^{-1}$ for all $v \in V$.
\end{lemma}

\begin{proof}
Note that $x$ being real implies $x^2 = e$. Since $x$ has no nonzero fixed point on $V$,  
by Proposition \ref{vector group}, for each $v \in V$ there exists a unique $w \in V$ such that 
$xv = w x w^{-1}.$
Since $x^2 = e$, we have 
$xv = w x^{-1} w^{-1} = (w x w^{-1})^{-1} = (xv)^{-1}.$
Thus $xv$ is real for all $v \in V$.
\end{proof}

In the following, we provide some examples in support of Lemma \ref{AV}.
\begin{example}
  Let $G=S^1\ltimes \mathbb R^2=
\left\{
\begin{pmatrix}
cos t & -sin t & x \\
sin t & cos t & y \\
0 & 0 & 1
\end{pmatrix}
: t, x, y \in \mathbb{R}
\right\}
$.  By Lemma \ref{AV}, we conclude that $(-I,n)$ is real for all $n\in \mathbb R^2.$ The elements $(I,n)$ are also real, and they are conjugate to their inverses by the diagonal matrix $(-1,-1,1)$.
\end{example}

The above observation shows that although the action of $I$ on $\mathbb{R}^2$ is trivial, $(I,n)$ is real for all $n \in \mathbb{R}^2$. This raises the question: Is the condition necessary for nontrivial elements in a solvable group? The following remark demonstrates that this can indeed occur for non-identity elements, showing that the condition in the theorem is sufficient but not necessary.

\begin{remark}\label{remark-sufficient}
This example shows that the condition in Lemma \ref{AV} is not necessary, even for non-identity elements. 
Consider 
\[
G = S^1 \ltimes \mathbb{R}^2 =
\left\{
\begin{pmatrix}
\cos(2t) & -\sin(2t) & x & 0 & 0 \\
\sin(2t) & \cos(2t)  & y & 0 & 0 \\
0 & 0 & 1 & 0 & 0 \\
0 & 0 & 0 & \cos t & -\sin t \\
0 & 0 & 0 & \sin t & \cos t
\end{pmatrix}
: t,x,y \in \mathbb{R}
\right\}.
\]
For $t = \pi$, the element
\[
A =
\begin{pmatrix}
1 & 0 & 0 & 0 & 0 \\
0 & 1 & 0 & 0 & 0 \\
0 & 0 & 1 & 0 & 0 \\
0 & 0 & 0 & -1 & 0 \\
0 & 0 & 0 & 0 & -1
\end{pmatrix}
\]
has order two and acts trivially on $\mathbb{R}^2$.  
For any $\begin{pmatrix} x \\ y \end{pmatrix} \in \mathbb{R}^2$, the element
\[
\begin{pmatrix}
1 & 0 & x & 0 & 0 \\
0 & 1 & y & 0 & 0 \\
0 & 0 & 1 & 0 & 0 \\
0 & 0 & 0 & -1 & 0 \\
0 & 0 & 0 & 0 & -1
\end{pmatrix}
\]
is conjugate to its inverse via
\[
\begin{pmatrix}
-1 & 0 & 0 & 0 & 0 \\
0 & -1 & 0 & 0 & 0 \\
0 & 0 & 1 & 0 & 0 \\
0 & 0 & 0 & 0 & -1 \\
0 & 0 & 0 & 1 & 0
\end{pmatrix}.
\]
\end{remark}

Proof of Theorem~\ref{a class of solvable groups}:
 $(1)$ Suppose $x \in A$ is conjugate to $x^{-1}$ in $G$.  
Then there exists $y \in G$ such that 
$x = y x^{-1} y^{-1}.$
Write $y = x_0 n$ with $x_0 \in A$ and $n \in N$. Then
$x = x_0 n x^{-1} (x_0 n)^{-1}
   = x^{-1} \big( x x_0 n x^{-1} n^{-1} x_0^{-1} \big),$
and hence
$x^2 = x_0 ((x n x^{-1}) (n^{-1})) x_0^{-1}.$ Thus, $x^2 \in N \cap A = \{e\}$, proving $(1).$

$(2)$ Since $xn$ is real implies $xN$ is real in $G/N$, it follows that $x$ is real in $A$, and hence $x^2=e.$ 

$(3)$ We proceed as in the proof of Theorem \ref{theorem-F-nilpotent}.  
Let $x \in A$ and $n \in N$, and consider the central series
$N = N_0 \triangleright N_1 \triangleright \cdots \triangleright N_r = \{e\}.$
In $\langle x \rangle \cdot N / N_1$, Lemma \ref{AV} gives $x n N_1 = (x n N_1)^{-1}$.  
Since $x^2 = e$ and $x_1 = w x w^{-1}$ also satisfies $x_1^2 = e$, repeating this step down the series yields
$x n = x_1 n_1 = \cdots =x_{r-1}n_{r-1}= x_r,$
with $x_r^2 = e$. Hence $(x n)^2 = e$, i.e., $xn = (x n)^{-1}$. Therefore, it is clear that $xn$ is real if and only if $xn$ is strongly real for all $n\in N$.


$(4)$ Since the action of $A$ on $Z(N)$ is trivial, and $Z(N)$ is a vector space, we conclude:
$xn$ is real in $AZ(N)$ if and only if $x$ is real in $A$ and $n$ is real in $Z(N)$.
Indeed, identifying $xn$ with $(x,n) \in AZ(N)$, suppose $(x,n)$ is real. Then there exists 
$(x_1,n_1) \in AZ(N)$ such that
$x_1 x x_1^{-1} = x^{-1},$ and $n_1 (x_1 n) (x^{-1} n_1^{-1}) = x^{-1} n^{-1}$ hold.
Since $A$ is abelian, the first condition gives $x = x^{-1}$, hence $x^2 = e$. Moreover, because the action 
of $A$ on $Z(N)$ is trivial, the second condition reduces to $n_1 n n_1^{-1} = n^{-1}.$
But as $n \in Z(N)$, we have $n_1 n n_1^{-1} = n$. Therefore, the above relation forces
$n = n^{-1}$, i.e. $n^2 = e.$
Since $N$ is a simply connected nilpotent Lie group, it contains no nontrivial finite-order elements. 
Hence $n = e$.
This proves $(4)$. \qed

\begin{remark}
Let $G$ be a connected solvable Lie group of the form $A\ltimes N$ as in Theorem \ref{a class of solvable groups}, and $x\in A$ be real. The following example shows that, even if the $x$-action is trivial on $Z(N)$ but non-trivial on other subquotients $N_j/N_{j+1}$ of $N$, $xn$ need not be real for all $n \in N\setminus Z(N)$. 

Let $
N = \left\{
\begin{pmatrix}
1 & a & c \\
0 & 1 & b \\
0 & 0 & 1
\end{pmatrix}
: a,b,c \in \mathbb{C}
\right\}$
be the $3$-dimensional complex Heisenberg group, and let \(A = \mathbb{C}^\times\) act on \(N\) by
$
\lambda \cdot
\begin{pmatrix}
1 & a & c \\
0 & 1 & b \\
0 & 0 & 1
\end{pmatrix}
=
\begin{pmatrix}
1 & \lambda a & c \\
0 & 1 & b/\lambda \\
0 & 0 & 1
\end{pmatrix},
\; \lambda \in \mathbb{C}^\times.
$
Set \(G = A \ltimes N\). 
Since $x\in A$ is real, we have $x^2 = 1$. 

Case 1: $x = 1$.  
Let $n = \begin{pmatrix}
1 & a & c \\
0 & 1 & b \\
0 & 0 & 1
\end{pmatrix} \notin Z(N), \;(a,b) \neq (0,0).$
If $a \neq 0$, $(1,n)$ is conjugate to $(1,n)^{-1}$ by $(-1,m)$ where
$m =
\begin{pmatrix}
1 & 0 & 0 \\
0 & 1 & \frac{ab - 2c}{a} \\
0 & 0 & 1
\end{pmatrix}.$
If $b \neq 0$, we may take
$m =
\begin{pmatrix}
1 & \frac{2c - ab}{b} & 0 \\
0 & 1 & 0 \\
0 & 0 & 1
\end{pmatrix}.$

Case 2: $x = -1$.  
Let $(\lambda,k) \in G$ with
$k =
\begin{pmatrix}
1 & p & z \\
0 & 1 & q \\
0 & 0 & 1
\end{pmatrix}.$
We require $k (\lambda n) (x^{-1} k^{-1}) = x^{-1} n^{-1}.$
The left-hand side is
$\begin{pmatrix}
1 & 2p + \lambda a & 2p q + c + \lambda a q + \frac{p b}{\lambda} \\
0 & 1 & 2q + \frac{b}{\lambda} \\
0 & 0 & 1
\end{pmatrix}$ and the right-hand side is
$\begin{pmatrix}
1 & a & ab - c \\
0 & 1 & b \\
0 & 0 & 1
\end{pmatrix}.$ Comparing entries gives
$p = \frac{a - \lambda a}{2}, \quad q = \frac{b - b/\lambda}{2}.$
Equality holds at $(1,3)$-position if $ab = 2c$, a condition independent of $\lambda$.  
Hence $(-1,n)$ is real exactly when $ab = 2c$.
\end{remark}
\section{Applications}

We first recall the following classification of real elements in ${\rm SL}(2,\mathbb{R})$.

\begin{lemma}\label{lem-rev-SL}
An element of ${\rm SL}(2,\mathbb{R})$ is real if and only if it is conjugate to
\[
\begin{pmatrix}
r & 0 \\
0 & r^{-1}
\end{pmatrix}, 
\quad r \in \mathbb{R} \setminus \{0\}.
\]
Moreover, elements that conjugate 
$\begin{pmatrix} r & 0 \\ 0 & r^{-1} \end{pmatrix}$ to its inverse are precisely those of the form 
$\begin{pmatrix} 0 & t \\ -t^{-1} & 0 \end{pmatrix}$.\qed
\end{lemma}

Proof of Theorem~\ref{them-homogeneous-rep}:
For $x = \begin{pmatrix} r & 0 \\ 0 & r^{-1} \end{pmatrix}$, the representation $\rho(x)$ is diagonal:
\[
\rho(x) = 
\begin{cases}
\mathrm{diag}(r^n, r^{n-2}, \dots, r^2, 1, r^{-2}, \dots, r^{-n}), & n \ \text{even},\\[2pt]
\mathrm{diag}(r^n, r^{n-2}, \dots, r, r^{-1}, \dots, r^{-n}), & n \ \text{odd}.
\end{cases}
\]

We now determine the real elements of ${\rm SL}(2,\mathbb{R}) \ltimes V_n$ by separately examining the cases when $n$ is even or odd.

 $(1)$ For even $n$:
we identify $(x,v) \in {\rm SL}(2,\mathbb{R}) \ltimes V_n$ with the matrix
$\begin{pmatrix}
\rho(x) & v\\
0 & 1
\end{pmatrix}.$

$(a)$ For $r = \pm 1$, $\rho(g) = I_{(n+1)\times (n+1)}$. The element $\begin{pmatrix} I_{(n+1)\times (n+1)} & v\\ 0 & 1 \end{pmatrix}$ is real if there exists $y \in {\rm SL}(2,\mathbb{R})$ such that $\rho(y)v = -v$. The proof of this is a straightforward computation.

 $(b)$ For $r \neq \pm 1$, we have
$\rho(x) = \mathrm{diag}(r^n, r^{n-2}, \dots, r^2, 1, r^{-2}, \dots, r^{-n}).$
Let $y = \begin{pmatrix} 0 & t \\ -t^{-1} & 0 \end{pmatrix}$, then $\rho(y)$ is an anti-diagonal matrix (from the top-right corner  to the bottom-left)  denoted by 
$\rho(y) =
\mathrm{antidiag}(t^n, t^{n-2}, \dots, t^2, -1, t^{-2}, \dots, t^{-n}).$
An element $(\rho(x),v)$ is real if there exists $w \in V_n$ satisfying
$-\rho(x)^{-1}w + w = -\rho(y)v - \rho(x)^{-1}v.$
Writing $w = (w_1,\ldots,w_{n+1})$ and $v = (v_1,\ldots,v_{n+1})$, this yields:
\begin{align*}
(-r^{-n} + 1) w_{1} &= -t^n v_{n+1} - r^{-n} v_1, \\ \vdots\\
(-r^{-2} + 1) w_{\frac{n}{2}} &= -t^2 v_{\frac{n+4}{2}} - r^{-2} v_{\frac{n}{2}},\\
0 &= v_{\frac{n+2}{2}} - v_{\frac{n+2}{2}}, \\
(-r^{2} + 1) w_{\frac{n+4}{2}} &= -t^{-2} v_{\frac{n}{2}} - r^{2} v_{\frac{n+4}{2}},\\ \vdots\\
(-r^n + 1) w_{n+1} &= -t^{-n} v_1 - r^{n} v_{n+1}.
\end{align*}
Since $r \neq \pm 1$, the system has an unique solution  for $w_1, w_2,\ldots, w_{\frac{n}{2}},  w_{\frac{n+4}{2}},\ldots,  w_{n+1}$ (upto a fixed $t$), while $w_{\frac{n+2}  {2}}$ is arbitrary. Thus $(x,v)$ is real for all $v \in V_n$. 

$(c)$ It is known that the only strongly real elements in $\mathrm{SL}(2,\mathbb R)$ are $\pm I$. In this case, $(\pm I,v)$ are strongly real elements if there exist involutions $y$ satisfying statement $\mathrm{(a)}$. Note that $\pm I$ are the only two involutions in $\mathrm{SL}_2(\mathbb R)$, and neither of them satisfies statement $\mathrm{(a)}$. Hence, $(\pm I,v) $ are not strongly real elements for $0\neq v\in V_n$.

\medskip

$(2)$ For odd $n$:

$(a)$ if $x = I_{2\times 2}$, then $(I_{(n+1)\times (n+1)},v)$ is real if there exists $y \in {\rm SL}(2,\mathbb{R})$ such that $\rho(y)v = -v$. It is easy to that for $y=-I_{2\times 2}$, $\rho(-I_{2\times 2})=-I_{n\times n}$.

$(b)$ for $x = \begin{pmatrix} r & 0 \\ 0 & r^{-1} \end{pmatrix}$, we have
$\rho(x) =
\mathrm{diag}(r^n, r^{n-2}, \dots, r, r^{-1}, \dots, r^{-n}).$
If $x \neq I_{2\times 2}$, then $\rho(x)$ fixes no nonzero vector in $V_n$, so $(x,v)$ is real by Proposition \ref{vector group}. 

$(c)$ In this case, $(\pm I,v)$ are strongly real elements for all $v$, since condition $\mathrm{(a)}$ is satisfied for the involution $y=-I$.

\qed


\begin{remark}
 Note that every real element in ${\rm SL}(2,\mathbb{R}) \ltimes V_n$ is necessarily rational.  
In particular, if $x \in {\rm SL}(2,\mathbb{R})$ is rational and $(x,v)$ is real, then $(x,v)$ is also rational.  
\end{remark}

We now turn to the proof of Corollary~\ref{cor-affine}, where we prove that if $x \in {\rm GL}(n,\mathbb{R})$ is rational, then $xv \in G={\rm GL}(n,\mathbb{R}) \ltimes \mathbb{R}^n$ is rational for every $v \in \mathbb{R}^n$.  
An analogous statement holds for ${\rm GL}(n,\mathbb{C}) \ltimes \mathbb{C}^n$, and we omit the details.

\vspace{.2cm}

\begin{corollary}\label{cor-affine}
Let $G = {\rm GL}(n,\mathbb{R}) \ltimes \mathbb{R}^n.$
If $x \in {\rm GL}(n,\mathbb{R})$ is rational, then $xv\in G$ is rational for all $v \in \mathbb{R}^n$.
\end{corollary}

\begin{proof}
We realise $G$ as $\{\begin{pmatrix}
x & v \\[2pt]
0 & 1
\end{pmatrix}|\;x\in {\rm GL}(n,\mathbb R), v\in\mathbb R^n\}.$
Therefore, for $x \in {\rm GL}(n,\mathbb{R})$ and $v \in \mathbb{R}^n$, we write 
 $xv\in G$ as $\begin{pmatrix}
x & v \\[2pt]
0 & 1
\end{pmatrix}.$

\textbf{Case 1:} Suppose the action of $x$ has no nonzero fixed point in $\mathbb{R}^n$.  
In this case, the corollary follows directly from Proposition \ref{vector group}.

\textbf{Case 2:} Suppose action of $x$ has a nonzero fix point in $\mathbb{R}^n$. That is, $1$ is an eigenvalue of $x$. Then we divide the proof into two sub-cases.

 $\it {Subcase \; 2.1}$ Suppose the order of $x$ is infinite. 
A straightforward computation shows that $(xv)^k\in G$ can be realised as
\[
\begin{pmatrix}
x^k & (x^{k-1} + x^{k-2} + \cdots +x+ I)v \\
0 & 1
\end{pmatrix}.
\]
Therefore, the order of $xv$ is also infinite. For an infinite order element, it follows that $x$ is rational if and only if $x$ is real. Hence, by Lemma 3.10 of \cite{GLM23} that $xv$ is real for all  $v\in\mathbb R^n$. Since $xv$ is of infinite order, we conclude that $xv$ is rational for all $v\in\mathbb R^n$.

$\it {Subcase \; 2.2}$ Suppose the order of $x$ is finite and equal to $m$, i.e., $x^m = I_{n\times n}$.  
Thus, the roots of the minimal polynomial of $x$ are distinct, being $m^{th}$ roots of unity. Without loss of generality, assume that the algebraic multiplicity of the eigenvalue $1$ is one, then the Jordan form of $x$, denoted by $\widetilde{J}(x)$, is \[
\begin{pmatrix}
    [1] & 0\\
    0 & J(x)_{(n-1)\times (n-1)}
\end{pmatrix},
\]
where $J(x)_{(n-1)\times (n-1)}$ is the diagonal block matrix whose diagonal blocks are the Jordan blocks of $x$ other than $[1]$. Clearly, if $x$ is rational, then $\widetilde{J}(x)$ is also rational.

We now show  $J(x)_{(n-1)\times (n-1)}$ is rational in ${\rm GL}(n-1,\mathbb R)$. Note that, the order of $J(x)_{(n-1)\times (n-1)}$ is $m$. 
Let $k$ be an integer with $\gcd(k,m)=1$. Since $\Tilde{J}(x)$ is rational, there exists $g=(g_{ij})_{n \times n} \in {\rm GL}(n,\mathbb R)$ such that $g\Tilde{J}(x)g^{-1}=\Tilde{J}(x)^k$, that is, $g\Tilde{J}(x)=\Tilde{J}(x)^kg$. Comparing the first column vector of both sides, we have \[ (g_{21},
    g_{31},\ldots, g_{n1})^T
= J(x)_{(n-1)\times (n-1)}^k(g_{21},
    g_{31},\ldots, g_{n1})^T,\]
Since  eigenvalue of $J(x)_{(n-1)\times (n-1)}$ is not equal 1, $(g_{21},
    g_{31},\ldots, g_{n1})^T=0$. Thus, $g$ has the form $\begin{pmatrix}
    g_{11} & *\\
    0 & g'
\end{pmatrix},$ where $g'$ is invertible. Comparing again in $g\Tilde{J}(x)=\Tilde{J}(x)^kg$, we deduce $g'J(x)_{(n-1)\times (n-1)}=J(x)_{(n-1)\times (n-1)}^kg'$. Hence $J(x)_{(n-1)\times (n-1)}$ is rational in ${\rm GL}(n-1,\mathbb R)$.

Now, we analyze the rationality of the elements $\begin{pmatrix}
    \Tilde{J}(x) & v\\
    0 & 1
\end{pmatrix}$ using the rationality of the element $J(x)_{(n-1)\times (n-1)}$.  For any $l\geq 2$, \[
\begin{pmatrix}
    \Tilde{J}(x) & v\\
    0 & 1
\end{pmatrix}^l =
\begin{pmatrix}
\Tilde{J}(x)^l & (\Tilde{J}(x)^{l-1} + \Tilde{J}(x)^{l-2} + \cdots + I)v \\
0 & 1
\end{pmatrix}.
\]
Let $v=(v_1, v_2, \ldots, v_n)^T$ and assume $v_1\neq 0$. Then, the first entry of the vector $(\Tilde{J}(x)^{l-1} + \Tilde{J}(x)^{l-2} + \cdots + I)v$ is equal to the non-zero number $lv_1$. Therefore, the order of $\begin{pmatrix}
    \Tilde{J}(x) & v\\
    0 & 1
\end{pmatrix}$ is infinite. By Lemma 3.10 of \cite{GLM23}, it is real, and hence rational. 
This also implies that if $\begin{pmatrix}
    \Tilde{J}(x) & v\\
    0 & 1
\end{pmatrix}$ has finite order, then $v_1=0$. In this case, consider the matrix $J(x)_{(n-1)\times(n-1)}$ and vector $v'=(v_2,v_3,\ldots, v_n)^T \in \mathbb{R}^{n-1}$. Since $J(x)_{(n-1)\times(n-1)}$ is rational in ${\rm GL}(n-1,\mathbb R)$ and none of its eigenvalues equal 1, the action of $J(x)_{(n-1)\times(n-1)}$ has no nonzero fix point in $\mathbb{R}^{n-1}$. Therefore, $\begin{pmatrix}
    J(x)_{(n-1)\times(n-1)} & v'\\
    0 & 1
\end{pmatrix}$ is rational, by Proposition \ref{vector group}.
Moreover, $\begin{pmatrix}
    J(x)_{(n-1)\times(n-1)} & v'\\
    0 & 1
\end{pmatrix}_{n\times n}$ and 
$\begin{pmatrix}
    \widetilde{J}(x) & v\\
    0 & 1
\end{pmatrix}_{(n+1)\times (n+1)}$ have the same order (say $m_1$). If $k_1$ is an integer with $\gcd(k_1,m_1)=1$, the rationality of $\begin{pmatrix}
    J(x)_{(n-1)\times(n-1)} & v'\\
    0 & 1
\end{pmatrix}$ ensures the existence of a $(n-1)\times (n-1)$ matrix $P'$ and vector $w'=(w_2, w_3,\ldots, w_n)^T$ such that $$\begin{pmatrix}
    P' & w'\\
    0 & 1
\end{pmatrix}\begin{pmatrix}
    J(x)_{(n-1)\times(n-1)} & v'\\
    0 & 1
\end{pmatrix}\begin{pmatrix}
    P' & w'\\
    0 & 1
\end{pmatrix}^{-1}=\begin{pmatrix}
    J(x)_{(n-1)\times(n-1)} & v'\\
    0 & 1
\end{pmatrix}^{k_1}.$$ Now, set $P=\begin{pmatrix}
    1 &0\\
    0 &P'
\end{pmatrix} \text{ and } w=\begin{pmatrix}
    0\\
    w'
\end{pmatrix}.$ A simple calculation shows that,  $$\begin{pmatrix}
    P & w\\
    0 & 1
\end{pmatrix}\begin{pmatrix}
    \Tilde{J}(x) & v\\
    0 & 1
\end{pmatrix}\begin{pmatrix}
    P & w\\
    0 & 1
\end{pmatrix}^{-1}=\begin{pmatrix}
    \Tilde{J}(x) & v\\
    0 & 1
\end{pmatrix}^{k_1}.$$

Note that in the proof above, we assume, for simplicity, that the algebraic multiplicity of the eigenvalue \(1\) is one of the element $x$. If the algebraic multiplicity is \(k>1\), the same computation shows that \(g\) has the form
$
g=
\begin{pmatrix}
g'' & *\\
0 & g'
\end{pmatrix},
$
where \(g''\) is an upper triangular matrix and \(g'\) is invertible. Moreover, in analyzing the rationality of the element
$
\begin{pmatrix}
\widetilde{J}(x) & v\\
0 & 1
\end{pmatrix},
$
we observe that its order is infinite whenever
$
(v_1,\ldots,v_k)\neq 0,
$ where $v_1, \ldots v_k$ are the first $k$-entries of $v$.
The remainder of the proof then proceeds exactly as above. Thus, to simplify the exposition, we assume that the algebraic multiplicity of the eigenvalue \(1\) is one.
This completes the proof. \end{proof}

\vspace{.2cm}

\noindent
{\bf Acknowledgement:}
The first-named author gratefully acknowledges S. G. Dani for a helpful comment on an earlier version, and the authors thank Arindam Jana for some discussions.
The second author acknowledges support from the NBHM Postdoctoral Fellowship (PDF no.\ 0204/27/(29)/2023/R\&D-II/11930).

\vspace{.2cm}
\noindent
{\bf Conflict of interest:}
On behalf of all authors, the corresponding author states that there is no conflict of interest.

 \bibliography{References.bib}

@article {AV,
    AUTHOR = {Amit, Alon and Vishne, Uzi},
     TITLE = {Characters and solutions to equations in finite groups},
   JOURNAL = {J. Algebra Appl.},
  FJOURNAL = {Journal of Algebra and its Applications},
    VOLUME = {10},
      YEAR = {2011},
    NUMBER = {4},
     PAGES = {675--686},
      ISSN = {0219-4988,1793-6829},
   MRCLASS = {20D60 (20F70)},
  MRNUMBER = {2834108},
       DOI = {10.1142/S0219498811004690},
       URL = {https://doi.org/10.1142/S0219498811004690},
}

@article {GM23,
    AUTHOR = {Gongopadhyay, Krishnendu and Maity, Chandan},
     TITLE = {Reality of unipotent elements in classical {L}ie groups},
   JOURNAL = {Bull. Sci. Math.},
  FJOURNAL = {Bulletin des Sciences Math\'ematiques},
    VOLUME = {185},
      YEAR = {2023},
     PAGES = {Paper No. 103261, 29},
      ISSN = {0007-4497,1952-4773},
   MRCLASS = {22E46 (17B08 20E45 22E60)},
  MRNUMBER = {4583694},
MRREVIEWER = {Sushil\ Bhunia},
       DOI = {10.1016/j.bulsci.2023.103261},
       URL = {https://doi.org/10.1016/j.bulsci.2023.103261},
}

@article {NT08,
    AUTHOR = {Navarro, Gabriel and Tiep, Pham Huu},
     TITLE = {Rational irreducible characters and rational conjugacy classes
              in finite groups},
   JOURNAL = {Trans. Amer. Math. Soc.},
  FJOURNAL = {Transactions of the American Mathematical Society},
    VOLUME = {360},
      YEAR = {2008},
    NUMBER = {5},
     PAGES = {2443--2465},
      ISSN = {0002-9947,1088-6850},
   MRCLASS = {20C15 (20C33 20E45)},
  MRNUMBER = {2373321},
MRREVIEWER = {Antonio\ Beltr\'an},
       DOI = {10.1090/S0002-9947-07-04375-9},
       URL = {https://doi.org/10.1090/S0002-9947-07-04375-9},
}

@article {DM17,
    AUTHOR = {Dani, S. G. and Mandal, Arunava},
     TITLE = {On the surjectivity of the power maps of a class of solvable
              groups},
   JOURNAL = {J. Group Theory},
  FJOURNAL = {Journal of Group Theory},
    VOLUME = {20},
      YEAR = {2017},
    NUMBER = {6},
     PAGES = {1089--1101},
      ISSN = {1433-5883,1435-4446},
   MRCLASS = {20D10},
  MRNUMBER = {3719318},
MRREVIEWER = {Ana\ Martinez-Pastor},
       DOI = {10.1515/jgth-2017-0013},
       URL = {https://doi.org/10.1515/jgth-2017-0013},
}

@article {GLM23,
    AUTHOR = {Gongopadhyay, Krishnendu and Lohan, Tejbir and Maity, Chandan},
     TITLE = {Reversibility of affine transformations},
   JOURNAL = {Proc. Edinb. Math. Soc. (2)},
  FJOURNAL = {Proceedings of the Edinburgh Mathematical Society. Series II},
    VOLUME = {66},
      YEAR = {2023},
    NUMBER = {4},
     PAGES = {1217--1228},
      ISSN = {0013-0915,1464-3839},
   MRCLASS = {20E45 (15B33 32M17 53A15)},
  MRNUMBER = {4679222},
MRREVIEWER = {Ian\ Short},
       DOI = {10.1017/s001309152300069x},
       URL = {https://doi.org/10.1017/s001309152300069x},
}

@book {OS15,
    AUTHOR = {O'Farrell, Anthony G. and Short, Ian},
     TITLE = {Reversibility in dynamics and group theory},
    SERIES = {London Mathematical Society Lecture Note Series},
    VOLUME = {416},
 PUBLISHER = {Cambridge University Press, Cambridge},
      YEAR = {2015},
     PAGES = {xii+281},
      ISBN = {978-1-107-44288-7},
   MRCLASS = {20-02 (37C85 37E99)},
  MRNUMBER = {3468569},
MRREVIEWER = {Peter\ Ha\"issinsky},
       DOI = {10.1017/CBO9781139998321},
       URL = {https://doi.org/10.1017/CBO9781139998321},
}

@article {TZ05,
    AUTHOR = {Tiep, Pham Huu and Zalesski, A. E.},
     TITLE = {Real conjugacy classes in algebraic groups and finite groups
              of {L}ie type},
   JOURNAL = {J. Group Theory},
  FJOURNAL = {Journal of Group Theory},
    VOLUME = {8},
      YEAR = {2005},
    NUMBER = {3},
     PAGES = {291--315},
      ISSN = {1433-5883,1435-4446},
   MRCLASS = {20D06 (20E45)},
  MRNUMBER = {2137972},
MRREVIEWER = {Wujie\ Shi},
       DOI = {10.1515/jgth.2005.8.3.291},
       URL = {https://doi.org/10.1515/jgth.2005.8.3.291},
}

@article {ST08,
    AUTHOR = {Singh, Anupam and Thakur, Maneesh},
     TITLE = {Reality properties of conjugacy classes in algebraic groups},
   JOURNAL = {Israel J. Math.},
  FJOURNAL = {Israel Journal of Mathematics},
    VOLUME = {165},
      YEAR = {2008},
     PAGES = {1--27},
      ISSN = {0021-2172,1565-8511},
   MRCLASS = {20G15 (20E45)},
  MRNUMBER = {2403612},
       DOI = {10.1007/s11856-008-1001-6},
       URL = {https://doi.org/10.1007/s11856-008-1001-6},
}

@article {SH08,
    AUTHOR = {Short, Ian},
     TITLE = {Reversible maps in isometry groups of spherical, {E}uclidean
              and hyperbolic space},
   JOURNAL = {Math. Proc. R. Ir. Acad.},
  FJOURNAL = {Mathematical Proceedings of the Royal Irish Academy},
    VOLUME = {108},
      YEAR = {2008},
    NUMBER = {1},
     PAGES = {33--46},
      ISSN = {1393-7197,2009-0021},
   MRCLASS = {57M10 (20B25)},
  MRNUMBER = {2448811},
MRREVIEWER = {John\ G.\ Ratcliffe},
       DOI = {10.3318/PRIA.2008.108.1.33},
       URL = {https://doi.org/10.3318/PRIA.2008.108.1.33},
}

@article {GL22,
    AUTHOR = {Gongopadhyay, Krishnendu and Lohan, Tejbir},
     TITLE = {Reversibility of {H}ermitian isometries},
   JOURNAL = {Linear Algebra Appl.},
  FJOURNAL = {Linear Algebra and its Applications},
    VOLUME = {639},
      YEAR = {2022},
     PAGES = {159--176},
      ISSN = {0024-3795,1873-1856},
   MRCLASS = {20E45 (15B33 15B57)},
  MRNUMBER = {4368550},
MRREVIEWER = {Roksana\ S\l owik},
       DOI = {10.1016/j.laa.2022.01.009},
       URL = {https://doi.org/10.1016/j.laa.2022.01.009},
}

@book {AA68,
    AUTHOR = {Arnol\cprime d, V. I. and Avez, A.},
     TITLE = {Ergodic problems of classical mechanics},
      NOTE = {Translated from the French by A. Avez},
 PUBLISHER = {W. A. Benjamin, Inc., New York-Amsterdam},
      YEAR = {1968},
     PAGES = {ix+286},
   MRCLASS = {28.70 (70.00)},
  MRNUMBER = {232910},
}

@article {D76,
    AUTHOR = {Devaney, Robert L.},
     TITLE = {Reversible diffeomorphisms and flows},
   JOURNAL = {Trans. Amer. Math. Soc.},
  FJOURNAL = {Transactions of the American Mathematical Society},
    VOLUME = {218},
      YEAR = {1976},
     PAGES = {89--113},
      ISSN = {0002-9947,1088-6850},
   MRCLASS = {58F05},
  MRNUMBER = {402815},
MRREVIEWER = {H.\ Suzuki},
       DOI = {10.2307/1997429},
       URL = {https://doi.org/10.2307/1997429},
}

@article {L92,
    AUTHOR = {Lamb, J. S. W.},
     TITLE = {Reversing symmetries in dynamical systems},
   JOURNAL = {J. Phys. A},
  FJOURNAL = {Journal of Physics. A. Mathematical and General},
    VOLUME = {25},
      YEAR = {1992},
    NUMBER = {4},
     PAGES = {925--937},
      ISSN = {0305-4470,1751-8121},
   MRCLASS = {58F37 (70H33)},
  MRNUMBER = {1151093},
MRREVIEWER = {Manfred\ Andri\'e},
       URL = {http://stacks.iop.org/0305-4470/25/925},
}

@article{pc,
    AUTHOR = {Chatterjee, P.},
    TITLE =  {On abstract homomorphisms of algebraic groups}, 
    JOURNAL = {https://arxiv.org/abs/1511.06180},
    YEAR = {2015},
}

@article {dix,
    AUTHOR = {Dixmier, J.},
     TITLE = {L'application exponentielle dans les groupes de {L}ie
              r\'esolubles},
   JOURNAL = {Bull. Soc. Math. France},
  FJOURNAL = {Bulletin de la Soci\'et\'e{} Math\'ematique de France},
    VOLUME = {85},
      YEAR = {1957},
     PAGES = {113--121},
      ISSN = {0037-9484},
   MRCLASS = {17.0X},
  MRNUMBER = {92930},
MRREVIEWER = {P.\ M.\ Cohn},
       URL = {http://www.numdam.org/item?id=BSMF_1957__85__113_0},
}

@book {kn,
    AUTHOR = {Knapp, Anthony W.},
     TITLE = {Lie groups beyond an introduction},
    SERIES = {Progress in Mathematics},
    VOLUME = {140},
   EDITION = {Second},
 PUBLISHER = {Birkh\"auser Boston, Inc., Boston, MA},
      YEAR = {2002},
     PAGES = {xviii+812},
      ISBN = {0-8176-4259-5},
   MRCLASS = {22-01},
  MRNUMBER = {1920389},
}

@article {kms,
    AUTHOR = {Kumar, Parteek and Mandal, Arunava and Singh, Shashank Vikram},
     TITLE = {On stable {C}artan subgroups of {L}ie groups},
   JOURNAL = {Bull. Sci. Math.},
  FJOURNAL = {Bulletin des Sciences Math\'ematiques},
    VOLUME = {212},
      YEAR = {2026},
     PAGES = {Paper No. 103857, 22},
      ISSN = {0007-4497,1952-4773},
   MRCLASS = {22E15 (17B40 22E60)},
  MRNUMBER = {5081859},
       DOI = {10.1016/j.bulsci.2026.103857},
       URL = {https://doi.org/10.1016/j.bulsci.2026.103857},
}

@article {I49,
    AUTHOR = {Iwasawa, Kenkichi},
     TITLE = {On some types of topological groups},
   JOURNAL = {Ann. of Math. (2)},
  FJOURNAL = {Annals of Mathematics. Second Series},
    VOLUME = {50},
      YEAR = {1949},
     PAGES = {507--558},
      ISSN = {0003-486X},
   MRCLASS = {20.0X},
  MRNUMBER = {29911},
MRREVIEWER = {R.\ Godement},
       DOI = {10.2307/1969548},
       URL = {https://doi.org/10.2307/1969548},
}

@article {Mo62,
    AUTHOR = {Mostow, G. D.},
     TITLE = {Covariant fiberings of {K}lein spaces. {II}},
   JOURNAL = {Amer. J. Math.},
  FJOURNAL = {American Journal of Mathematics},
    VOLUME = {84},
      YEAR = {1962},
     PAGES = {466--474},
      ISSN = {0002-9327,1080-6377},
   MRCLASS = {22.70 (55.99)},
  MRNUMBER = {142688},
       DOI = {10.2307/2372983},
       URL = {https://doi.org/10.2307/2372983},
}
\bibliographystyle{amsalpha}
\end{document}